\documentclass{amsart}
\usepackage{geometry}
\usepackage{amsmath}
\usepackage{hyperref}
\usepackage[alphabetic]{amsrefs}
\usepackage{amssymb}

\newtheorem{proposition}{Proposition}[section]
\newtheorem{theorem}[proposition]{Theorem}
\newtheorem{lemma}[proposition]{Lemma}
\newtheorem{corollary}[proposition]{Corollary}
\newtheorem{conjecture}[proposition]{Conjecture}
\newtheorem{question}[proposition]{Question}

\theoremstyle{definition}
\newtheorem{remark}[proposition]{Remark}
\newtheorem{definition}[proposition]{Definition}

\newcommand{\bR}{\mathbb{R}}
\newcommand{\bA}{\mathbb{A}}
\newcommand{\bQ}{\mathbb{Q}}
\newcommand{\bZ}{\mathbb{Z}}
\newcommand{\bN}{\mathbb{N}}

\newcommand{\cX}{\mathcal{X}}
\newcommand{\cY}{\mathcal{Y}}
\newcommand{\cZ}{\mathcal{Z}}
\newcommand{\cO}{\mathcal{O}}
\newcommand{\cL}{\mathcal{L}}
\newcommand{\cI}{\mathcal{I}}

\newcommand{\cF}{\mathcal{F}}

\newcommand{\cE}{\mathcal{E}}

\newcommand{\fR}{\mathfrak{R}}

\newcommand{\rd}{\mathrm{d}}

\newcommand{\hbeta}{\widehat{\beta}}

\newcommand{\lct}{\mathrm{lct}}

\newcommand{\vol}{\mathrm{vol}}
\newcommand{\ord}{\mathrm{ord}}

\newcommand{\Dtc}{\Delta_{\mathrm{tc}}}
\newcommand{\Fut}{\mathrm{Fut}}
\newcommand{\tc}{{\rm {tc}}}

\title{Some criteria for uniform K-stability}
\author{Chuyu Zhou}
\address{Beijing International Center for Mathematical Research,Peking University, Beijing 100871, China}
\email{chuyuzhou@pku.edu.cn}

\author{Ziquan Zhuang}
\address{Department of Mathematics, Princeton University, Princeton, NJ 08544, USA.}
\email{zzhuang@math.princeton.edu}

\date{} 

\begin{document}
\begin{abstract}
We prove some criteria for uniform K-stability of log Fano pairs. In particular, we show that uniform K-stability is equivalent to $\beta$-invariant having a positive lower bound. Then we study the relation between optimal destabilization conjecture and the conjectural equivalence between uniform K-stability and K-stability in twisted setting.
\end{abstract}

\maketitle
\tableofcontents

\section{Introduction}\label{section1}

K-stability is an important concept introduced in \cite{Tian97} (and later algebraically reformulated in \cite{Don02}) to test whether there is a KE metric on a projective Fano manifold (see in particular \cites{Tian15, CDS15a, CDS15b, CDS15c}). However, it's difficult to check K-stability of a Fano manifold and various equivalent but simpler criteria have been introduced in terms of special test configurations \cite{LX14}, valuations and filtrations \cites{Fuj19,Li17} and stability thresholds (or $\delta$-invariants) \cites{FO18,BJ17}.






In this note, we give some more criteria for uniform K-stability from these perspectives. Since uniform K-stability has certain openness property, i.e. K-semistability is preserved after small perturbation of the boundary divisor (see \cite{Fuj17}), we first have the following criterion (note that the direction $(1)\Rightarrow (2)$ has been known by \cite{Fuj17}).

\begin{theorem}[=Theorem \ref{uks via boundary}] \label{2}
Let $(X,\Delta)$ be a log Fano pair. The following are equivalent:
\begin{enumerate}
\item $(X,\Delta)$ is uniformly K-stable.

\item There exists a $\epsilon>0$ such that $(X,\Delta+\epsilon D)$ is K-semistable for any $D\in|-K_X-\Delta|_\mathbb{R}$.
\end{enumerate}
\end{theorem}

Our next criterion gives a way to test uniform K-stability using only $\beta$-invariant (see Section \ref{section2} for related definitions): 

\begin{theorem}\label{3}
Let $(X,\Delta)$ be a log Fano pair. The following are equivalent:
\begin{enumerate}
\item $(X,\Delta)$ is uniformly K-stable. 

\item There exists $\epsilon>0$ such that $\beta_{X,\Delta}(E)\geq \epsilon$ for any divisor $E$ over X. 

\item There exists $\epsilon>0$ such that $\beta_{X,\Delta}(E)\geq \epsilon$ for any dreamy divisor $E$ over X. 

\item There exists $\epsilon>0$ such that $\beta_{X,\Delta}(E)\geq \epsilon$ for any weakly special divisor $E$ over X. 
\end{enumerate}
\end{theorem}

It is well expected that K-stability is equivalent to uniform K-stability. This statement is proved to be equivalent to the existence of divisorial valuation computing $\delta$-invariant when $\delta(X,\Delta)=1$ (see Section \ref{section5}). In \cite{BLZ19}, an algebraic twisted K-stability theory is developed to study $\bQ$-Fano varieties that are not uniformly K-stable. We introduce concepts of twisted K-stability and twisted uniform K-stability and similarly expect they are equivalent. We then explore the relation between this equivalence 
and the existence of divisorial valuation computing $\delta$-invariant when $\delta(X,\Delta)<1$. In particular, we prove:

\begin{theorem}[=Theorem \ref{maximal twist}]\label{4}
Let $(X,\Delta)$ be a log Fano pair with $\delta(X,\Delta)\leq 1$, then for any $0<\mu<\delta(X,\Delta)$, $(X,\Delta)$ is $\mu$-twisted uniformly K-stable. Besides, $(X,\Delta)$ is $\delta(X,\Delta)$-twisted K-semistable but not $\delta(X,\Delta)$-twisted uniformly K-stable.
\end{theorem}

This is a refinement of the twisted valuative criterion established in \cite{BLZ19}. Using this result, we establishes the equivalence between the existence of divisorial $\delta$-minimizer and the conjecture ``K-stable = Uniformly K-stable'' in the twisted setting.

\begin{theorem}[=Theorem \ref{equivalence}]\label{5}
For a log Fano pair $(X,\Delta)$ with $\delta(X,\Delta)\leq1$, that twisted K-stable is equivalent to twisted uniformly K-stable, is equivalent to the existence of a divisorial valuation computing $\delta(X,\Delta)$.
\end{theorem}

The paper is organized as follows. In Section \ref{section2}, we recall the notion and some preliminaries that will be used later. In Section \ref{section3}, we prove the criteria for uniform K-stability, i.e.  Theorems \ref{2} and \ref{3}. In Section \ref{section4}, we introduce the concept of twisted K-stability and twisted uniform K-stability and prove Theorem \ref{4}. In Section \ref{section5}, we prove Theorem \ref{5}.\\

\noindent
\subsection*{Acknowledgement}
The first author would like to thank his advisor Chenyang Xu for his constant support and encouragement. The second author would like to thank his advisor J\'anos Koll\'ar for constant support and encouragement. Both author would like to thank Yuchen Liu for helpful discussions and for suggesting them to write down this note and Harold Blum for useful discussion. Most of this work is done during both authors' visit at MSRI, which is gratefully acknowledged.

\section{Notion and preliminaries}\label{section2}

We work over $\mathbb{C}$. We refer to \cites{KM98,Kollar13} for the definition of singularities of pairs. A projective normal variety $X$ is called $\bQ$-Fano if $-K_X$ is an ample $\mathbb{Q}$-Cartier divisor and $X$ admits klt singularities. A pair $(X,\Delta)$ is called log Fano if $-K_X-\Delta$ is an ample $\mathbb{Q}$-Cartier divisor and $(X,\Delta)$ is klt. The $\bR$-linear system of an $\bR$-Cartier $\bR$-divisor $L$ is defined to be $|L|_\bR = \{D\ge 0\,|\,D\sim_\bR L\}$. Similar one can define the $\bQ$-linear system $|L|_\bQ$ of a $\bQ$-Cartier $\bQ$-divisor.

\subsection{Test configurations}

Let $(X,\Delta)$ be a log Fano pair. A test configuration $(\cX,\Dtc;\cL)$ of $(X,\Delta;-K_X-\Delta)$ consists of the following data:

\begin{enumerate}

 \item A projective morphism $\pi: \mathcal{X}\rightarrow \mathbb{A}^1$ and an effective $\bQ$-divisor $\Delta_{\tc}$ on $\mathcal{X}$.

\item A relatively ample $\mathbb{Q}$-line bundle $\mathcal{L}$ on $\mathcal{X}$.

\item A $\mathbb{C}^*$-action on $(\mathcal{X},\Delta_{\tc};r\mathcal{L})$ for some sufficiently divisible integer $r$ such that $(\mathcal{X}^*,\Delta_{\tc}^*;r\cL|_{\cX^*})$ is $\mathbb{C}^*$-equivariantly isomorphic to $(X,\Delta;-r(K_X+\Delta))\times (\mathbb{A}^1\setminus 0)$ via the projection $\pi$, where $\mathcal{X}^*=\mathcal{X}\setminus \mathcal{X}_0$ and $\Delta_{\tc}^*={\Delta_{\tc}}|_{\mathcal{X}^*}$.


\end{enumerate}

Unless otherwise specified, all test configurations considered in this note are assumed to be normal, i.e. $\cX$ is normal in the above definition. One can glue $(\mathcal{X},\Dtc)$ and $(X,\Delta)\times (\mathbb{P}^1\setminus 0)$ along their common open subset $(X,\Delta)\times (\mathbb{A}^1 \setminus 0)$ to get a natural compactification $(\overline{\mathcal{X}},\overline{\Delta}_{\tc};\overline{\mathcal{L}})$.
A test configuration $(\mathcal{X},\Delta_{\tc};\mathcal{L})$ is called special (resp. weakly special) if $(\mathcal{X},\Delta_{\tc}+\mathcal{X}_0)$ is \emph{plt} (resp. \emph{lc}) and $\mathcal{L}\sim_\mathbb{Q}-K_{\mathcal{X}/\mathbb{A}^1}-\Delta_{\tc}$.

A test configuration $(\mathcal{X},\Delta_{\tc};\mathcal{L})$ is trivial if it's $\mathbb{C}^*$-equivariant to $((X,\Delta)\times \mathbb{A}^1,-K_{X\times \mathbb{A}^1/\mathbb{A}^1})$. It is said to be of product type if it's induced by a diagonal $\mathbb{C}^*$-action on $(X,\Delta)\times \mathbb{A}^1$ given by a one parameter subgroup of ${\rm Aut}(X)$.

\subsection{K-stability}


Given a test configuration $(\mathcal{X},\Delta_{\tc};\mathcal{L})$ of an $n$-dimensional log Fano pair $(X,\Delta)$, its \emph{generalized Futaki invariant} is defined as follows:

$$\Fut(\mathcal{X},\Delta_{\tc};\mathcal{L}):=\frac{n \overline{\cL}^{n+1}}{(n+1)(-K_X-\Delta)^n}+\frac{\overline{\cL}^n\cdot (K_{\overline{\cX}/\mathbb{P}^1}+\overline{\Delta}_{\tc})}{(-K_X-\Delta)^n}. $$

We say $(X,\Delta)$ is \emph{K-semistable} if $\Fut(\mathcal{X},\Delta_{\tc};\mathcal{L})\geq 0$ for any normal test configuration $(\mathcal{X},\Delta_{\tc};\mathcal{L})$. 
We say $(X,\Delta)$ is \textit{K-stable} if $\Fut(\mathcal{X},\Delta_{\tc};\mathcal{L})> 0$ for any non-trivial normal test configuration $(\mathcal{X},\Delta_{\tc};\mathcal{L})$.
We say $(X,\Delta)$ is \textit{K-polystable} if it is K-semistable and $\Fut(\mathcal{X},\Delta_{\tc};\mathcal{L})> 0$ for any non product type normal test configuration $(\mathcal{X},\Delta_{\tc};\mathcal{L})$.

To define uniform K-stability, we introduce $J$-invariant of a test configuration $(\mathcal{X},\Delta_{\tc};\mathcal{L})$ as follows \cites{BHJ17, Fuj19}:
$$J(\mathcal{X},\Delta_{\tc};\mathcal{L}):= \frac{\Pi^*(-K_{X\times \mathbb{P}^1/\mathbb{P}^1}-\Delta_{\mathbb{P}^1})^n\cdot\Theta^*\overline{\cL}}{(-K_X-\Delta)^n}-\frac{\overline{\cL}^{n+1}}{(n+1)(-K_X-\Delta)^n},$$
where $\Pi: \mathcal{Z}\rightarrow X\times \mathbb{P}^1$ and $\Theta: \mathcal{Z}\rightarrow \mathcal{X}$ denote the normalization of the graph of $X\times \mathbb{P}^1\dashrightarrow \mathcal{X}$.

We say $(X,\Delta)$ is \textit{uniformly K-stable} if there is a positive number $0<\epsilon<1$ such that $\Fut(\mathcal{X},\Delta_{\tc};\mathcal{L})\geq \epsilon J(\mathcal{X},\Delta_{\tc};\mathcal{L})$ for any normal test configuration.

\subsection{Dreamy divisor and special divisor}

In this subsection, we introduce two kinds of divisors which will appear frequently later.

Let $(X,\Delta)$ be a log Fano pair. We say $E$ is a divisor over $X$ if there is a birational model $\sigma:Y\rightarrow X$ such that $E$ is a prime divisor on $Y$. If $E\subset X$ we just let $\sigma={\rm id}_X$.

\begin{definition}[\cite{Fuj19}]
We say that $E$ is a dreamy divisor or $\ord_E$ is a dreamy valuation over $X$ if $\bigoplus_{i,j\in\bN} H^0(X,-ir\sigma^*(K_X+\Delta)-jE)$ is finitely generated, where $r$ is a positive integer such that $-r(K_X+\Delta)$ is Cartier.
\end{definition}

\begin{definition}
We say that $E$ is a (weakly) special divisor or $\ord_E$ is a (weakly) special valuation over $X$ if 
it's induced by a non-trivial (weakly) special test configuration $(\mathcal{X},\Delta_{\tc};\mathcal{L})$, i.e. $\ord_E$ is proportional to the restriction of $\ord_{\mathcal{X}_0}$ (since $\ord_{\mathcal{X}_0}$ is a divisorial valuation on the function field $K(\mathcal{X})=K(X\times \mathbb{A}^1)$, we just restrict the valuation to $K(X)$ to get a divisorial valuation over $X$; see \cite{BHJ17}).
\end{definition}

We have the following characterization of dreamy divisors (see \cite[Theorem 5.1]{Fuj19} and \cite[Lemma 3.8]{Fuj17b}).

\begin{lemma}
If E is a dreamy divisor over X, then there is a test configuration $(\mathcal{X},\Delta_{\tc};\mathcal{L})$ whose central fiber is integral such that $\ord_E$
is proportional to the restriction of $\ord_{\mathcal{X}_0}$. Conversely, if $(\mathcal{X},\Delta_{\tc};\mathcal{L})$ is a test configuration whose central fiber is integral, then the restriction of $\ord_{\mathcal{X}_0}$ is a dreamy valuation over X.
\end{lemma}

\begin{remark}
If $(\mathcal{X},\Delta_{\tc};\mathcal{L})$ is a test configuration whose central fiber is integral, then $\mathcal{L}\sim_{\mathbb{Q}}-K_{\mathcal{X}/\mathbb{A}^1}-\Delta_{\tc}$.
\end{remark}

\subsection{Various invariants}

In this subsection, we recall the $\beta$-invariants and $\delta$-invariants of log Fano pairs.

\begin{definition}[\cite{Fuj19}]
Let $(X,\Delta)$ be a log Fano pair and $E$ a divisor over $X$. The $\beta$-invariant of $E$ (or $\ord_E$) is defined as:
$$\beta_{X,\Delta}(E):=A_{X,\Delta}(E)-\frac{1}{(-K_X-\Delta)^n} \int_0^\infty\vol(-K_X-\Delta-xE)\rd x.$$
Note that the above definition differs from Fujita's original definition by a multiple. We also write $$S_{X,\Delta}(E):=\frac{1}{(-K_X-\Delta)^n} \int_0^\infty\vol(-K_X-\Delta-xE)\rd x$$
and let $T_{X,\Delta}(E)$ be the pseudo-effective threshold of $-E$ with respect to $-K_X-\Delta$, i.e.
$$T_{X,\Delta}(E)=\sup\{x\in\bR^+\,|\,\vol(-K_X-\Delta-xE)>0\}.$$
Finally we let $j_{X,\Delta}(E)=T_{X,\Delta}(E)-S_{X,\Delta}(E)$.
\end{definition}

\begin{remark}\label{remark of relation}
We have the following relation between $S_{X,\Delta}(E)$ and $T_{X,\Delta}(E)$ (see e.g. \cite[Lemma 2.6]{BJ17}):
$$\frac{1}{n+1}T_{X,\Delta}(E)\leq S_{X,\Delta}(E)\leq\frac{n}{n+1}T_{X,\Delta}(E) .$$
It then follows that
$$\frac{1}{n}S_{X,\Delta}(E)\leq j_{X,\Delta}(E)\leq nS_{X,\Delta}(E) .$$
\end{remark}

$\beta$-invariant has a close relation to K-stability, as discovered in \cite{Fuj19} and \cite{Li17} (see also \cite{BX18} for part of the statement):

\begin{theorem} \label{Fujita}
Let $(X,\Delta)$ be a log Fano pair. The following are equivalent:
\begin{enumerate}
\item $(X,\Delta)$ is K-semistable $($resp. K-stable, uniformly K-stable$)$.
 
\item $\beta_{X,\Delta}(E)\geq 0$ $($resp. $>0$, $\ge \epsilon j_{X,\Delta}(E)$ for some fixed $\epsilon>0)$ for any divisorial valuation ${\rm{ord}}_E$ over $X$.

\item $\beta_{X,\Delta}(E)\geq 0$ $($resp. $>0$, $\ge \epsilon j_{X,\Delta}(E)$ for some fixed $\epsilon>0)$ for any dreamy divisorial valuation ${\rm{ord}}_E$ over $X$.

\item $\beta_{X,\Delta}(E)\geq 0$ $($resp. $>0$, $\ge \epsilon j_{X,\Delta}(E)$ for some fixed $\epsilon>0)$ for any special divisorial valuation ${\rm{ord}}_E$ over $X$.










\end{enumerate}
\end{theorem}



The following $\delta$-invariant is introduced by \cite{FO18} to characterize K-stability.

\begin{definition}
Let $(X,\Delta)$ be a log Fano pair and $E$ a divisor over $X$. We set
$$ \delta_m(X,\Delta) :=  \sup \left\{a\in \mathbb{R}^+ | (X,\Delta+aD_m) \text{ is lc for any $m$-basis type divisor }D_m \right\} ,$$
and
$$\delta(X,\Delta):=\limsup_m \delta_m(X,\Delta) .$$
By \cite{BJ17}, the above limsup is in fact a limit and we have
$$\delta(X,\Delta)=\inf_E\frac{A_{X,\Delta}(E)}{S_{X,\Delta}(E)} ,$$
where the infimum runs over all divisorial valuations $E$ over $X$.
\end{definition}

Parallel to Theorem \ref{Fujita} we have (\cites{FO18,BJ17,Fuj19}):

\begin{theorem}\label{parallel criterion}
Let $(X,\Delta)$ be a log Fano pair. The following are equivalent:
\begin{enumerate}
\item $(X,\Delta)$ is K-semistable $($resp. K-stable, uniformly K-stable$)$.

\item $\frac{A_{X,\Delta}(E)}{S_{X,\Delta}(E)}\geq 1$ $($resp. $>1$, $>1+\epsilon$ for some fixed $\epsilon>0)$ for any divisorial valuation ${\rm{ord}}_E$ over $X$. 

\item $\frac{A_{X,\Delta}(E)}{S_{X,\Delta}(E)}\geq 1$ $($resp. $>1$, $>1+\epsilon$ for some fixed $\epsilon>0)$ for any dreamy divisorial valuation ${\rm{ord}}_E$ over $X$. 

\item $\frac{A_{X,\Delta}(E)}{S_{X,\Delta}(E)}\geq 1$ $($resp. $>1$, $>1+\epsilon$ for some fixed $\epsilon>0)$ for any special divisorial valuation ${\rm{ord}}_E$ over $X$. 







 


\end{enumerate}
\end{theorem}








\section{Criteria for K-stability}\label{section3}

In this section, we will establish several criteria for uniform K-stability.

\begin{theorem}\label{uks via boundary}
Suppose $(X,\Delta)$ is a log Fano pair, then the following two are equivalent:
\begin{enumerate}
\item $(X,\Delta)$ is uniformly K-stable.

\item There exists some $\epsilon>0$ such that $(X,\Delta+\epsilon D)$ is K-semistable for any $D\in|-K_X-\Delta|_\mathbb{R}$.
\end{enumerate}
\end{theorem}

\begin{proof}
For any divisorial valuation ${\rm{ord}}_E$ over $X$, 
$$\beta_{X,\Delta+\epsilon D}(E)=A_{X,\Delta}(E)-\epsilon\cdot{\rm{ord}}_E(D)-(1-\epsilon) S_{X,\Delta}(E).$$

Assume (2) holds, then $\beta_{X,\Delta+\epsilon D}(E)\ge 0$ for all $D\in |-K_X-\Delta|_\mathbb{R}$ and all divisors $E$ over $X$. Taking the supremum over $D$ we have
$$A_{X,\Delta}(E)-\epsilon T_{X,\Delta}(E)-(1-\epsilon) S_{X,\Delta}(E)\geq 0, $$
i.e.  
$$A_{X,\Delta}(E)-S_{X,\Delta}(E)\geq \epsilon \cdot(T_{X,\Delta}(E)-S_{X,\Delta}(E))=\epsilon\cdot j_{X,\Delta}(E),$$
which implies (1).

Conversely, suppose $(X,\Delta)$ is uniformly K-stable, then by Theorem \ref{Fujita}, there exists some $\mu$ with $0<\mu<1$, such that $\frac{A_{X,\Delta}(E)}{S_{X,\Delta}(E)}\geq 1+\mu$ for any divisorial valuation ${\rm{ord}}_E$ over $X$. By Remark \ref{remark of relation}, we can choose a $0<\epsilon<1$ such that 
$$(1+\mu)S_{X,\Delta}(E)\geq \epsilon\cdot T_{X,\Delta}(E)+(1-\epsilon) S_{X,\Delta}(E),$$
say $\epsilon=\frac{\mu}{n+1}$. Thus
\begin{align*}
\beta_{X,\Delta+\epsilon\cdot D}(E) & = A_{X,\Delta}(E)-\epsilon\cdot {\rm{ord}}_E(D)-(1-\epsilon) S_{X,\Delta}(E) \\
 & \geq  A_{X,\Delta}(E)-\epsilon\cdot T_{X,\Delta}(E)-(1-\epsilon) S_{X,\Delta}(E)\geq 0
\end{align*}
for any $D\in |-K_X-\Delta|_\mathbb{R}$. So, $(X,\Delta+\epsilon D)$ is K-semistable by Theorem \ref{Fujita}(1).
\end{proof}

Inspired by the above Theorem \ref{uks via boundary}, we can define a new invariant for a log Fano pair $(X,\Delta)$, the \textit{uniformity} of $(X,\Delta)$, which characterizes how uniformly K-stable $(X,\Delta)$ is.

\begin{definition}
Suppose $(X,\Delta)$ is a given K-semistable log Fano pair. The uniformity of $(X,\Delta)$ is defined as follows:
$$u(X,\Delta):=\sup\left\{a\in\mathbb{R}_{\geqslant0}| \mbox{$(X,\Delta+aD)$ is K-semistable} ,  \forall D\in|-K_X-\Delta|_\mathbb{R}\right\} .$$
\end{definition}

We can give a precise characterization for $u(X,\Delta)$.

\begin{proposition}
Let $(X,\Delta)$ be a K-semistable klt log Fano pair, then $$u(X,\Delta)=\inf_E\frac{\beta_{X,\Delta}(E)}{j_{X,\Delta}(E)},$$
where E runs through all divisors over X.
\end{proposition}

\begin{proof}
Suppose $a$ is a nonnegative real number such that $(X,\Delta+a D)$ is K-semistable for any $D\in|-K_X-\Delta|_\mathbb{R}$, then we have 
$$\beta_{X,\Delta+a D}(E)=A_{X,\Delta}(E)- {\rm ord}_E(aD)-(1-a)S_{X,\Delta}(E)\geq0, $$
$\forall D\in|-K_X-\Delta|_\mathbb{R}$ and $\forall E$ over $X$. This is equivalent to
$$A_{X,\Delta}(E)-S_{X,\Delta}(E)\geq a(T_{X,\Delta}(E)-S_{X,\Delta}(E))$$
for any $E$ over $X$, i.e. 
$$a\leq \inf_E\frac{\beta_{X,\Delta}(E)}{j_{X,\Delta}(E)}. $$
\end{proof}

By Theorem \ref{uks via boundary}, we have the following corollary:

\begin{corollary}
Suppose $(X,\Delta)$ is a K-semistable log Fano pair. Then 
\begin{enumerate}
\item $(X,\Delta)$ is uniformly K-stable if and only if $u(X,\Delta)>0$.
\item $\delta(X,\Delta)=1$ if and only if $u(X,\Delta)=0$.
\end{enumerate}
\end{corollary}

\begin{remark}
By Remark \ref{remark of relation} we have the following relation between $u(X,\Delta)$ and $\delta(X,\Delta)-1$:
$$\frac{1}{n}(\delta(X,\Delta)-1)\leq u(X,\Delta)\leq n(\delta(X,\Delta)-1) .$$
\end{remark}

Theorem \ref{Fujita}, Theorem \ref{parallel criterion} and Theorem \ref{uks via boundary} give three characterizations of uniform K-stability. We now give another criterion using only $\beta$-invariant.

\begin{theorem}\label{thm:uKs via beta}
Let $(X,\Delta)$ be a log Fano pair. The following three are equivalent:
\begin{enumerate}
\item $(X,\Delta)$ is uniformly K-stable. 

\item There exists $\epsilon>0$ such that $\beta_{X,\Delta}(E)\geq \epsilon$ for any divisor $E$ over X. 

\item There exists $\epsilon>0$ such that $\beta_{X,\Delta}(E)\geq \epsilon$ for any dreamy divisor $E$ over X. 
\end{enumerate}
\end{theorem}

\begin{proof}
$(1)\Rightarrow (2)$: If $(X,\Delta)$ is uniformly K-stable, then there exists some $\delta>1$ such that $A_{X,\Delta}(E)\ge \delta\cdot S_{X,\Delta}(E)$ for all divisor $E$ over $X$. Since $(X,\Delta)$ is log Fano, we have $A_{X,\Delta}(E)\ge \frac{1}{r}$ where $r$ is an integer such that $r(K_X+\Delta)$ is Cartier. Thus
\[
\beta_{X,\Delta}(E)=A_{X,\Delta}(E)-S_{X,\Delta}(E)\ge (1-\delta^{-1})A_{X,\Delta}(E)\ge \frac{1-\delta^{-1}}{r}
\]
for any divisor $E$ over $X$ and we may simply take $\epsilon=\frac{1-\delta^{-1}}{r}>0$.

$(2)\Rightarrow (1)$: Suppose that $\beta_{X,\Delta}(E)\geq \epsilon>0$ for all divisor over $X$. By \cite[Corollary 3.6]{BJ17}, there exists a sequence $c_m$ ($m=1,2,\cdots$) of numbers depending only on $(X,\Delta)$ such that $\lim_{m\to \infty} c_m=1$ and $c_m\cdot \ord_E(D_m)\le S_{X,\Delta}(E)$ for any $m\in \bN$, any divisor $E$ over $X$ and all $m$-basis type divisor $D_m\sim_\bQ -(K_X+\Delta)$. It follows that
\[
A_{X,\Delta+c_m D_m}(E)=A_{X,\Delta}(E)-c_m\cdot \ord_E(D_m)\ge A_{X,\Delta}(E)-S_{X,\Delta}(E)=\beta_{X,\Delta}(E)\ge \epsilon
\]
for all $m$, $E$ and $D_m$ as above. In other words, the pair $(X,\Delta+c_m D_m)$ is $\epsilon$-lc. By \cite[Theorem 1.6]{Birkar-2} (applied to the pair $(X,B=\Delta+c_m D_m)$, $M=D_m$ and the very ample divisor $A=-r(K_X+\Delta)$ for some sufficiently large and divisible $r$), there exists some $t>0$ depending only on $(X,\Delta)$ such that $\lct(X,B;D_m)\ge t$ for all $m$ and $D_m$. Hence $(X,\Delta+(c_m+t)D_m)$ is lc for all $m\in \bN$ and all $m$-basis type divisor $D_m$, which implies $\delta_m(X,\Delta)\ge c_m+t$. Letting $m\to \infty$ we see that $\delta(X,\Delta)\ge 1+t>1$ and therefore $(X,\Delta)$ is uniformly K-stable.

$(3)\Leftrightarrow (2)$: One direction is obvious. For the other direction, note that by Theorem \ref{Fujita}, (3) implies that $(X,\Delta)$ is K-semistable, hence it suffices to show that if $(X,\Delta)$ is a K-semistable log Fano pair, then any divisor $E$ over $X$ for which $\beta_{X,\Delta}(E)<1$ is dreamy. This is proved in Lemma \ref{lem:dreamy}.
\end{proof}


\begin{lemma} \label{lem:dreamy}
Let $(X,\Delta)$ be a K-semistable log Fano pair and $E$ a divisor over $X$. Suppose that $\beta_{X,\Delta}(E)<1$. Then $E$ is dreamy.
\end{lemma}

\begin{proof}
By \cite[Lemma 3.5 and Corollary 3.6]{BJ17}, there exists $m$-basis type divisors $D_m\sim_\bQ -(K_X+\Delta)$ ($m\in\bN$) such that $\ord_E(D_m)\to S_{X,\Delta}(E)$ ($m\to \infty$). Let $\lambda_m=\min\{\delta_m(X,\Delta),1\}$. Since $(X,\Delta)$ is K-semistable, we have $\lim_{m\to \infty} \lambda_m =1$ and $(X,\Delta+\lambda_m D_m)$ is lc for all $m\in\bN$. Then as
$$A_{X,\Delta+\lambda_m D_m}(E)=A_{X,\Delta}(E) - \lambda_m \ord_E(D_m) \to \beta_{X,\Delta}(E)<1\,\,(m\to \infty),$$
we see that $A_{X,\Delta+\lambda_m D_m}(E)<1$ for $m\gg 0$. By \cite[Corollary 1.4.3]{BCHM10}, one can extract $E$ as a prime divisor on a Fano type variety and in particular $E$ is dreamy.
\end{proof}





In general, there are many dreamy divisors over a log Fano pair. We now show that those with small $\beta$-invariants are weakly special. In particular, combining with Theorem \ref{thm:uKs via beta}, this completes the proof of Theorem \ref{3}.

\begin{theorem}\label{weakly special}
Let $(X,\Delta)$ be a K-semistable log Fano pair. Then there exists
some $0 < \epsilon_0\ll 1$ such that
any dreamy divisor E over X with $\beta_{X,\Delta}(E) < \epsilon_0$  induces a weakly special test configuration of $(X, \Delta)$ with integral central fiber.
\end{theorem}

\begin{proof}
Let $\mathfrak{R}\subset [0, 1]$ be a finite set of rational numbers containing 1 and all finite sums of the coefficients of $\Delta$. Choose $\epsilon_0\in \bQ \bigcap (0, 1)$ such that a pair $(Y,B + G)$ (where $G$ is a reduced divisor and $\dim Y \leq \dim X + 1)$ is lc as long as $(Y, B + (1 -\epsilon_0)G)$ is lc and the coefficients of $B$ belongs to $\mathfrak{R}$. Such $\epsilon_0$ exists by the ACC of log canonical threshold \cite{HMX14}. Suppose $E$ is a divisor over $X$ with $\beta_{X,\Delta}(E)<\epsilon_0$, then similar to the proof of Lemma \ref{lem:dreamy} we can find a $D\in |-K_X-\Delta|_\bQ$ such that $(X,\Delta+D)$ is klt and $A_{X,\Delta+D}(E)<\epsilon_0$. By \cite[Corollary 1.4.3]{BCHM10}, one can extract $E$ on a birational model of $X$, say $\mu: Y\rightarrow X$ and 
$$K_Y+\widetilde{D}+\widetilde{\Delta}+cE=\mu^*(K_X+\Delta+D),$$ where $\widetilde{D}$ and $\widetilde{\Delta}$ are strict transformation of $D$ and $\Delta$ respectively and $1-\epsilon_0<c<1$. Note that $Y$ is of Fano type.
Consider the pair $(X_{\bA^1}, \Delta_{\bA^1}+D_{\bA^1}+X_0)$ (where $X_{\bA^1}=X\times \bA^1$, etc. and $X_0=X\times \{0\}$) which is a plt pair. Then there is an induced morphism $\mu_{\bA^1}: Y_{\bA^1}\rightarrow X_{\bA^1}$. Let $v$ be a quasi-monomial valuation over $X_{\bA^1}$ with weight $(1,1)$ along the divisors $X_0$ and $E_{\bA^1}$. It's clear that $v$ is a divisorial valuation over $X_{\bA^1}$ whose center is contained in $X_0$. Denote by $\cE$ the corresponding divisor over $X_{\bA^1}$, then $A_{X_{\bA^1},\Delta_{\bA^1}+D_{\bA^1}+X_0}(\cE)=A_{X,\Delta+D}(E)<\epsilon_0<1$, hence by \cite[Corollary 1.4.3]{BCHM10} we can extract $\cE$ on a projective birational model $\pi:\cY\to X_{\bA^1}$ of $X_{\bA^1}$. We have 
$$K_{\cY}+\pi_*^{-1}\Delta_{\bA^1}+\pi_*^{-1}D_{\bA^1}+\widetilde{X}_0+c\cE=\pi^*(K_{X_{\bA^1}}+\Delta_{\bA^1}+D_{\bA^1}+X_0),$$
where $\widetilde{X}_0$ is the strict transformation of $X_0$, $c>1-\epsilon_0$ and $\cY$ is of Fano type over $\bA^1$. Run the $\widetilde{X}_0$-MMP/$\bA^1$ on $\cY$, we get a minimal model $\cY\dashrightarrow \cY'$, and $\widetilde{X}_0$ is contracted by the negativity lemma (see e.g \cite[Lemma 3.39]{KM98}). Let $\Delta_{\bA^1}'$, $D_{\bA^1}'$ and $\cE'$ be the pushforward of $\pi_*^{-1}\Delta_{\bA^1}$, $\pi_*^{-1}D_{\bA^1}$ and $\cE$ on $\cY'$ respectively, then we know $\cY' \rightarrow \bA^1$ has an integral central fiber and the restriction of $\ord_{\cE'}$ is exactly $\ord_E$.
Let $(\cX,\Delta_{\tc};\cL)$ be the test configuration induced by the dreamy divisor $E$, then
$\cX$ is the ample model of $\cY'$ over $\bA^1$ with respect to $-(K_{\cY'}+\Delta_{\bA^1}')$. As $(\cY,\pi_*^{-1}\Delta_{\bA^1}+\pi_*^{-1}D_{\bA^1}+c\cE)$ is a klt Calabi-Yau pair (pullback of $(X_{\bA^1},\Delta_{\bA^1}+D_{\bA^1})$), we know that $(\cY',\Delta_{\bA^1}'+D_{\bA^1}'+c\cE')$ is also klt and the same holds for its strict transform on $\cX$. It follows that $(\cX,\Delta_{\tc}+c\cX_0)$ is a klt pair. As $c>1-\epsilon_0$, we see that $(\cX, \Delta_{\tc}+\cX_0)$ is lc by our choice of $\epsilon_0$.
\end{proof}

\begin{remark}
The above theorem says the following two statements are equivalent:
\begin{enumerate}
\item $(X,\Delta)$ is uniformly K-stable.
\item There is a $\epsilon>0$ such that $\beta_{X,\Delta}(E)\ge \epsilon$ for any weakly special divisor $E$ over $X$.
\end{enumerate}
Compared with Theorems \ref{Fujita} and \ref{parallel criterion}, one would expect that for uniform K-stability it's sufficient to check $\beta(E)\ge \epsilon$ for all special divisors $E$ over $X$, although this doesn't seem to follow from our current proof.
\end{remark}

It's expected that uniformly K-stable and K-stable are the same for any given log Fano pair. One direction is clear.  Assume $(X,\Delta)$ is K-stable, to confirm uniform K-stability, it suffices to show that there is a $\epsilon >0$ such that $\beta_{X,\Delta}(E)>\epsilon$ for any weakly special divisor $E$ over $X$. Let $\epsilon_0$ be as in the proof of Theorem \ref{weakly special}. Our next result (inspired by the recent work \cite{Xu19}) shows that it suffices to consider those $E$ that are bounded in some sense (note that a more general version that applies to all weakly special divisor is independently proved in \cite[Theorem A.2]{BLX19} using a somewhat different method):

\begin{theorem}\label{lc complement}
Let $(X,\Delta)$ be a K-semistable log Fano pair. If $E$ is a divisor over X with $\beta_{X,\Delta}(E)<\epsilon_0$, then we can find a $G\in \frac{1}{N}|-N(K_X+\Delta)|$ such that E is a lc place of $(X,\Delta+G)$. Here N is a positive integer number which only depends on $(X,\Delta)$.
\end{theorem}

\begin{proof}
As in the proof of Theorem \ref{weakly special}, we can find a $D\in |-K_X-\Delta|_{\bQ}$ such that $(X,\Delta+D)$ is klt and $A_{X,\Delta+D}(E)< \epsilon_0$. In addition, we can extract $E$ to be a divisor on a projective birational model of $X$, say $\mu: Y\rightarrow X$ and 
$$K_Y+\widetilde{\Delta}+\widetilde{D}+cE=\mu^*(K_X+\Delta+D) ,$$
where $\widetilde{\Delta}$ and $\widetilde{D}$ are the strict transformations and $1-\epsilon_0< c< 1$. Note that $Y$ is of Fano type, then we can run MMP for $-(K_Y+\widetilde{\Delta}+E)$. Suppose we get a Mori fiber space $Y\dashrightarrow Y'\rightarrow T$ and write $\widetilde{\Delta}'$ and $E'$ for the pushforward of $\widetilde{\Delta}$ and $E$ on $Y'$, then we know $(K_Y'+\widetilde{\Delta}'+E')|_F$ is ample where $F$ is the general fiber of $Y'\rightarrow T$. As $(Y,\widetilde{\Delta}+\widetilde{D}+cE)$ is a klt Calabi-Yau pair, so is $(Y',\widetilde{\Delta}'+\widetilde{D}'+cE')$. It follows that $(K_Y'+\widetilde{\Delta}'+cE')|_F$ is anti-nef and $(Y',\widetilde{\Delta}'+cE')$ is klt, thus $(Y',\widetilde{\Delta}'+E')$ is lc by the choice of $\epsilon_0$. But this contradicts \cite[Theorem 1.5]{HMX14}. So the MMP above produces a minimal model $Y\rightarrow Y'$ and $-(K_{Y'}+\widetilde{\Delta}'+E')$ is semiample.

Now by the boundedness of
complement \cite[Theorem 1.7]{Birkar-1}, there exists some integer $N > 0$ depending only
on the dimension and the set $\fR$ such that if $(Y', \widetilde{\Delta}'+E')$ is an lc pair of dimension $n$ with
coefficients in $\fR$, $Y'$ is of Fano type and $-(K_{Y'} + \widetilde{\Delta}'+E')$ is nef, then there exists some
effective divisor $G' \in \frac{1}{N}| -N(K_{Y'} +\widetilde{\Delta}'+E')|$ such that $(Y,\widetilde{\Delta}'+E'+ G')$ is lc.  It follows that $E$ is a lc place of the lc pair $(X,\Delta+G)$ where $G\in \frac{1}{N}|-N(K_X+\Delta)|$ is the pushforward of $G'$ to $X$.
\end{proof}

It is therefore very natural to ask the following question:

\begin{question}
Given a set $S$ of lc log Calabi-Yau pairs $(X,\Delta+D)$ such that $(X,\Delta)$ is log Fano. Let $S'$ be the set of lc log Calabi-Yau pairs that can be realized as special degenerations of pairs in $S$. Assume that $S$ is bounded. Is $S'$ bounded?
\end{question}

In particular, a positive answer to this question will lead to a proof that K-stability is equivalent to uniform K-stability (since the Futaki invariants have a bounded denominator in a bounded family). We don't know any proof or counterexample to the above question.

\begin{remark}
Theorem \ref{lc complement} also gives an approximation for $\delta(X,\Delta)=1$ using lc places of bounded lc complements, i.e. if $\delta(X,\Delta)=1$, then $\delta(X,\Delta)=\inf_{E} \frac{A_{X,\Delta}(E)}{S_{X,\Delta}(E)}$, where $E$ is a lc place of $(X,\Delta+G)$ for some lc $N$-complement $G$ of $(X,\Delta)$. See \cite[Corollary 3.6]{BLX19} for a more general statement when $\delta(X,\Delta)\leq 1$.
\end{remark}

\section{Twisted setting}\label{section4}

In this section, we will define K-stability in the twisted setting. To make it simple, we leave out the boundary as it doesn't play essential roles. $X$ always denotes a $\bQ$-Fano variety with $\delta(X)\leq 1$. We first recall the definition of twisted K-stability \cites{Der16, BLZ19}.

\begin{definition}
Let $(\mathcal{X},\mathcal{L})$ be a given normal test configuration of $X$, $0<\mu\leq 1$, then $\mu$-twisted generalized Futaki invariant is defined to be
$${\rm Fut}_{1-\mu}(\mathcal{X},\mathcal{L}):=\sup_{D\in |-K_X|_\mathbb{Q}}{\rm Fut}(\mathcal{X},(1-\mu)\mathcal{D};\mu\mathcal{L})$$
where $\mathcal{D}$ is closure of $D\times (\bA^1\setminus 0)$ in $\cX$.
\end{definition}

\begin{definition}
\begin{enumerate}

\item We say X is $\mu$-twisted K-semistable if ${\rm Fut}_{1-\mu}(\mathcal{X},\mathcal{L})\geq 0$ for every normal test configuration $(\mathcal{X},\mathcal{L})$.

\item We say X is $\mu$-twisted K-stable if ${\rm Fut}_{1-\mu}(\mathcal{X},\mathcal{L})> 0$ for every non-trivial normal test configuration $(\mathcal{X},\mathcal{L})$.

\item We say X is $\mu$-twisted uniformly K-stable if there exists a positive real number $\epsilon>0$ such that ${\rm Fut}_{1-\mu}(\mathcal{X},\mathcal{L})\geq \epsilon J(\mathcal{X},\mathcal{L})$ for every normal test configuration $(\mathcal{X},\mathcal{L})$.

\end{enumerate}
\end{definition}

In the above definition,  one should check all normal test configurations to test twisted K-stability. However, by a special test configuration theory in twisted setting that has been established in \cite[Theorem 1.6]{BLZ19} which is parallel to \cite{LX14}, we have following theorem:

\begin{theorem}
To test $\mu$-twisted K-semistability $($resp. K-stability and uniform K-stability$)$, it suffices to check all special test configurations.
\end{theorem}

While $X$ may not be K-semistable, it can still be K-stable in the twisted sense \cite{BLZ19}. The following result is a refinement of the twisted valuative criterion established in \cite[Theorem 1.5]{BLZ19}.

\begin{theorem}\label{maximal twist}
Let X be a $\bQ$-Fano variety with $\delta(X)\le 1$, then X is $\mu$-twisted uniformly K-stable for $0<\mu<\delta(X)$, and X is $\mu$-twisted K-semistable but not $\mu$-twisted uniformly K-stable for $\mu=\delta(X)$.
\end{theorem}

\begin{proof}
For $\mu<\delta(X)$, by \cite[Theorem C]{BL18b}, there is a $D\in |-K_X|_\mathbb{Q}$ such that $(X,(1-\mu)D)$
is uniformly K-stable. Thus there is a positive real number $0<\epsilon<1$ such that 
$${\rm Fut}(\mathcal{X},(1-\mu)\mathcal{D};\mathcal{L})\geq \epsilon  J(\mathcal{X},\mathcal{L})$$
for any normal test configuration, so one has
$${\rm Fut}_{1-\mu}(\mathcal{X},\mathcal{L})\geq \epsilon  J(\mathcal{X},\mathcal{L})$$

For $\mu=\delta(X)$, we can choose a sequence of special test configurations $(\mathcal{X}_i,\mathcal{L}_i)$ such that $\lim_i\frac{A(v_{\mathcal{X}_{i,0}})}{S(v_{\mathcal{X}_{i,0}})}=\delta(X)$ \cite[Theorem 4.3]{BLZ19}. We aim to prove that $X$ is not $\delta$-twisted uniformly K-stable where $\delta=\delta(X)$. If not, there is a positive real number $0<\epsilon<1$ such that 
$${\rm Fut}_{1-\delta}(\mathcal{X}_i,\mathcal{L}_i)\geq \epsilon J(\mathcal{X}_i,\mathcal{L}_i).$$
One can choose a general $D\in |-K_X|_\mathbb{Q}$ such that
$${\rm Fut}_{1-\delta}(\mathcal{X}_i,\mathcal{L}_i) ={\rm Fut}(\mathcal{X}_i,(1-\delta)\mathcal{D};\mathcal{L}_i)$$
for any $i$, where $D$ doesn't contain any center of $v_{\mathcal{X}_{i,0}}$ \cite[Theorem 3.7]{BLZ19}.
Thus one obtain
$${\rm Fut}_{1-\delta}(\mathcal{X}_i,\mathcal{L}_i)=A(v_{\mathcal{X}_{i,0}})-\delta S(v_{\mathcal{X}_{i,0}})\geq \epsilon J(\mathcal{X}_i,\mathcal{L}_i)=\epsilon (T(v_{\mathcal{X}_{i,0}})-S(v_{\mathcal{X}_{i,0}})),$$
which contradicts $\lim_i \frac{A(v_{\mathcal{X}_{i,0}})}{S(v_{\mathcal{X}_{i,0}})}=\delta(X)$.
\end{proof}

\section{Optimal Destabilization Conjecture}
\label{section5}

It has long been expected that uniform K-stability is equivalent to K-stability. In \cite{BX18}, they reduced the problem to the existence of divisorial $\delta$-minimizer for $\delta(X)=1$, that is, the divisorial valuation computing $\delta$-invariant. The algebraic twisted K-stability theory has been established to study K-unstable Fano varieties \cite{BLZ19},   then the case $\delta<1$ can be studied in parallel to the case $\delta=1$. In this section, we will explain the relation between the following two conjectures.

\begin{conjecture}{\rm (Optimal Destabilization Conjecture)}\label{optimal <1}
Let X be a $\bQ$-Fano variety with $\delta(X)\leq 1$, then there exists a divisor $E$ over $X$ computing $\delta(X)$, i.e.  $\frac{A(E)}{S(E)}=\delta(X).$
\end{conjecture}

\begin{conjecture}\label{tuks=tks}
Let X be a $\bQ$-Fano variety with $\delta(X)\leq 1$, and $0<\mu\leq 1$, then X is $\mu$-twisted K-stable is equivalent to that X is $\mu$-twisted uniformly K-stable.
\end{conjecture}

By Theorem \ref{maximal twist}, we first have the following lemma as a direct corollary:

\begin{lemma}
Conjecture \ref{tuks=tks} is equivalent to that X is not $\delta(X)$-twisted K-stable.
\end{lemma}


The above two conjectures are equivalent by following result:

\begin{theorem}\label{equivalence}
Conjecture \ref{optimal <1} is equivalent to Conjecture \ref{tuks=tks}.
\end{theorem}

\begin{proof}
We first assume Conjecture \ref{optimal <1}, i.e. there is a divisor $E$ computing $\delta=\delta(X)$, then by \cite[Theorem 1.1]{BLZ19}, $E$ is a dreamy divisor over $X$ which naturally induces a non-trivial test configuration $(\mathcal{X},\mathcal{L})$ such that ${\rm Fut}_{1-\delta}(\mathcal{X},\mathcal{L})=0$, thus $X$ is not $\delta$-twisted K-stable. Conversely, assume $X$ is not $\delta$-twisted K-stable, then there exists a non-trivial test configuration $(\mathcal{X},\mathcal{L})$ such that ${\rm Fut}_{1-\delta}(\mathcal{X},\mathcal{L})=0$. By \cite[Theorem 3.9]{BLZ19},  it must be a special test configuration  whose central fiber induces a divisorial valuation computing $\delta(X)$.
\end{proof}

We can also translate optimal destabilization conjecture into vanishing of $\delta$-twisted generalized Futaki invariant \cite{BLZ19}.

\begin{theorem}
Suppose X is a klt Fano variety with $\delta(X)\leq 1$. If there is a divisor E over X computing $\delta(X)$, i.e $\frac{A(E)}{S(E)}=\delta(X)$, then there is a test configuration $(\mathcal{X},\mathcal{L})$ such that ${\rm Fut}_{1-\delta}(\mathcal{X},\mathcal{L})=0$. Conversely, if there is a test configuration $(\mathcal{X},\mathcal{L})$ such that ${\rm Fut}_{1-\delta}(\mathcal{X},\mathcal{L})=0$, then there is a divisor E over X computing $\delta(X)$.
\end{theorem}

\begin{proof}
Suppose there is a divisor $E$ computing $\delta(X)$, then $E$ is a dreamy divisor which naturally induces a test configuration whose $\delta$-twisted generalized Futaki is zero, by \cite[Theorem 1.1]{BLZ19}. Conversely, if there is a test configuration whose $\delta$-twisted generalized Futaki is zero, then it must be a special test configuration whose central fiber induces a divisorial valuation computing $\delta(X)$, by \cite[Theorem 4.6]{BLZ19}.
\end{proof}

\begin{remark}
The first two conjectures in this section for $\delta(X)=1$ correspond to the following two conjectures:
\begin{enumerate}
\item {\rm (Optimal Destabilization Conjecture for $\delta=1$)}
Suppose $X$ is a $\bQ$-Fano variety with $\delta(X)=1$, then there is a divisorial valuation ${\rm{ord}}_E$ computing $\delta(X)$, i.e. $\delta(X)=\frac{A(E)}{S(E)}=1$.
\item For Fano varieties, uniform K-stability is equivalent to K-stability.
\end{enumerate}
By Theorem \ref{equivalence}, we know they are also equivalent.
\end{remark}

\bibliography{reference.bib}

\end{document}